\pdfoutput=1
\documentclass[11pt, reqno]{amsart}
\usepackage[dvips]{graphicx}
\usepackage{bmpsize} 
\usepackage[margin=1.12in]{geometry}

\usepackage{mathrsfs}
\usepackage{amsmath,amscd}
\usepackage{amssymb}
\usepackage{amscd}
\graphicspath{/home/jhois/Documents/Images/}  

\newtheorem{theorem}{Theorem}[section]
\newtheorem{proposition}[theorem]{Proposition}
\newtheorem{lemma}[theorem]{Lemma}

\newtheorem{definition}[theorem]{Definition}

\theoremstyle{definition} 

\newtheorem{remark}[theorem]{Remark}

\newcommand{\C}{\mathbb C} 
\newcommand{\R}{\mathbb R}

\DeclareMathOperator{\sys}{{\rm sys}}

\numberwithin{equation}{section} 
\numberwithin{theorem}{section}
\numberwithin{figure}{section}

\begin{document}

\author[J.A.~Hoisington]{Joseph Ansel Hoisington} \address{Max Planck Institute for Mathematics}\email{hoisington@mpim-bonn.mpg.de}

\title[Energy-minimizing Maps]{Energy-minimizing Mappings of Complex Projective Spaces}

\keywords{Energy-minimizing maps, infima of energy functionals in homotopy classes, harmonic and pluriharmonic mappings}
\subjclass[2020]{Primary 53C43, 53C55 Secondary 53C35}

\begin{abstract}
We show that in all homotopy classes of mappings from complex projective space to Riemannian manifolds, the infimum of the energy is proportional to the infimal area in the homotopy class of mappings of the 2-sphere which represents the induced homomorphism on the second homotopy group.  We then establish a family of optimal lower bounds for a larger class of energy functionals for mappings from real and complex projective space to Riemannian manifolds and characterize the mappings which attain these lower bounds.  
\end{abstract}

\maketitle

%%%%%%%%%%%%%%%%%%%%%%%%%%%%%%%%%%%%%%%%%%%%%%%%%%%%%%%%%%%%%%%%%%%%%%%%%%%%%%%%%%%%%%%%%%%%%%%%%%%%%%%%%%%%%%%%%%%%%%%%%

\section{Introduction} 
\label{introduction} 

This paper's first result determines the infimum of the energy in a homotopy class of mappings from complex projective space to a Riemannian manifold: 

\begin{theorem}
\label{infimum}

Let $(\C P^{N},g_{0})$ be complex projective space, with its canonical Riemannian metric normalized so the maximum of its sectional curvature is $4$.  Let $\Phi$ be a homotopy class of mappings from $(\C P^{N},g_{0})$ to a Riemannian manifold $(M,g)$ and $\varphi$ the homotopy class of mappings from the $2$-sphere to $(M,g)$ represented by composing the inclusion $\C P^{1} \subseteq \C P^{N}$ with $F \in \Phi$.  Let $A^{\star}$ be the infimal area of mappings in $\varphi$; that is: 

\begin{equation}
\label{cpn_invariant}
\displaystyle A^{\star} = \inf\limits_{f \in \varphi} \int\limits_{S^{2}} \sqrt{\det(df^{T} \circ df)} dA. 
\end{equation}

Then, letting $E_{2}(F)$ be the energy of a map $F$ and $C_{N} = \frac{\pi^{N-1}}{(N-1)!}$, 

\begin{equation}
\label{infimum_equation}
\displaystyle \inf\limits_{F \in \Phi} E_{2}(F) = C_{N} A^{\star}. 
\end{equation}   
\end{theorem}

After proving Theorem \ref{infimum}, we will discuss conditions under which this infimum is realized by a continuous map $F \in \Phi$.  One source of examples is the theorem of Lichnerowicz \cite{Li1} that holomorphic and antiholomorphic maps between compact K\"ahler manifolds minimize energy in their homotopy class---one can show via a direct calculation that any such map of complex projective space has energy equal to the infimal value in (\ref{infimum_equation}), cf. Proposition \ref{holomorphic_corollary}.  We will explain in Remark \ref{pluriharmonic_remark} how, for maps of complex projective space that minimize energy in their homotopy class, the following theorem of Ohnita leads to a partial converse to Lichnerowicz's result:  

\begin{theorem}{\em (Ohnita \cite{Oh1})}
\label{ohnita_theorem}
A stable harmonic map $F:(\C P^{N},g_{0}) \rightarrow (M,g)$ from complex projective space to a Riemannian manifold is pluriharmonic; that is, letting $\alpha_{F}$ be the second fundamental form of $F$ (as in Definition \ref{second_fund_form_defn}) and $J$ the complex structure of $\C P^{N}$, $\alpha_{F}(JV,JW) = -\alpha_{F}(V,W)$. 
\end{theorem} 

The determination of the infimum in Theorem \ref{infimum} builds on work of Croke, who showed in \cite{Cr1} that the energy of a map from complex projective space to a Riemannian manifold is bounded below by the value $C_{N}A^{\star}$ in (\ref{infimum_equation}), by showing that this lower bound is the infimum in all homotopy classes.  We will extend Croke's lower bound for energy to a lower bound for a larger class of functionals:  

\begin{theorem}
\label{cpn_p_energy_thm} 
Let $F:(\C P^{N},g_{0}) \rightarrow (M,g)$ be a Lipschitz map from complex projective space to a Riemannian manifold, $E_{p}(F)$ its $p$-energy (as in Definition \ref{p_energy_def}), and $A^{\star}$ the invariant associated to the homotopy class of $F$ in Theorem \ref{infimum}.  Then for all $p > 2$,  

\begin{equation}
\label{cpn_p_energy_thm_eqn}
\displaystyle E_{p}(F) \geq \frac{\pi^{N}}{2 N!} \left( \frac{2N}{\pi} A^{\star} \right)^{\frac{p}{2}}. \smallskip
\end{equation}

Equality for at least one $p > 2$ implies that $F$ is a homothety onto a pluriharmonically immersed minimal submanifold, and thus that equality holds in (\ref{cpn_p_energy_thm_eqn}) for all $p \geq 2$. 
\end{theorem} 

The results of White \cite{Wh1} imply that for $1 \leq p < 2$, the infimum of the $p$-energy in any homotopy class of mappings from complex projective space to a Riemannian manifold is $0$.  In this sense, the lower bound for $p$-energy for $p \geq 2$ given by Theorems \ref{infimum} and \ref{cpn_p_energy_thm} is optimal.  Special cases of Theorems \ref{infimum} and \ref{cpn_p_energy_thm} state that the identity mapping of complex projective space minimizes $p$-energy in its homotopy class for $p \geq 2$.  For $p > 2$, Theorem \ref{cpn_p_energy_thm} implies that the only $p$-energy-minimizing maps in its homotopy class are isometries, however for $p = 2$, Lichnerowicz's results \cite{Li1} cited above imply that all holomorphic maps of complex projective space minimize energy in their homotopy class.  In the homotopy class of the identity, this includes projective linear transformations which are not isometries. \\  

White's results in \cite{Wh1} also imply that the infimum of the energy is $0$ in all homotopy classes of mappings from quaternionic projective space, the Cayley projective plane, or the sphere of dimension $3$ or greater to a Riemannian manifold.  (For a more general statement proven in \cite{Wh1} from which this follows, see Remark \ref{homotopy_dependence}.)  Together, these results and Theorem \ref{infimum} determine the infimum of the energy in all homotopy classes of mappings of compact rank-$1$ Riemannian symmetric spaces other than real projective space, and in particular for all simply-connected spaces in this family.  For real projective space, we will establish some results which are similar to our results for complex projective space above, but we will also show that there are some potentially noteworthy differences. \\  

Croke showed via the results in \cite{Cr1} that the identity mapping of real projective space minimizes energy in its homotopy class.  We will extend these results to the following lower bound for the entire class of $p$-energy functionals:  

\begin{theorem}{\em (\cite{Cr1} for $p=2$)}
\label{rpn_p_energy_thm}
Let $(\R P^{n},g_{0})$, $n \geq 2$, be real projective space with its Riemannian metric of constant curvature $1$.  Let $F:(\R P^{n},g_{0}) \rightarrow (M,g)$ be a Lipschitz map, $E_{p}(F)$ its $p$-energy, $\gamma$ a simple closed geodesic in $(\R P^{n},g_{0})$, and $L^{\star}$ the infimum of the lengths of closed curves freely homotopic to $F_{*}(\gamma)$.  Then for all $p \geq 1$, letting $\sigma(n)$ be the volume of the unit $n$-sphere, 

\begin{equation}
\label{rpn_thm_eqn} 
\displaystyle E_{p}(F) \geq \frac{\sigma(n)}{4} \left( \frac{\sqrt{n}}{\pi} L^{\star} \right)^{p}. \smallskip 
\end{equation}

Equality for at least one $p > 1$ implies that $F$ is a homothety onto a totally geodesic submanifold, and thus that equality holds in (\ref{rpn_thm_eqn}) for all $p \geq 1$.  If $F$ is a smooth immersion, equality for $p=1$ also implies this. \end{theorem}

Theorem \ref{rpn_p_energy_thm} implies that the identity mapping of real projective space minimizes $p$-energy in its homotopy class for all $p \geq 1$, as the identity mapping of complex projective space minimizes $p$-energy in its homotopy class for $p \geq 2$.  Unlike complex projective space however, in which the energy corresponding to $p = 2$ admits a larger family of minimizing maps than the $p$-energy for $p > 2$, in real projective space the characterization of $p$-energy-minimizing maps is equally rigid for all $p \geq 1$.  For $p > 1$, the rigidity of the family of minimizing maps in Theorem \ref{rpn_p_energy_thm} can be derived from the result for $p=2$, however this argument does not address the case $p=1$.  The argument we give to address this case draws on the proof of the Blaschke conjecture by Berger and Kazdan, cf. \cite{Be1}. \\ 

Although Theorem \ref{rpn_p_energy_thm} gives a lower bound for energy for maps of real projective space which is similar to the lower bound for complex projective space implied by Theorem \ref{infimum}, determining the infimum of the energy in a homotopy class may be more complicated in this case.  First, consider this problem for the real projective plane: let $g$ be a Riemannian metric on $\R P^{2}$, of area $A(\R P^{2},g)$, and let $\psi$ be the homotopy class of the identity mapping of $\R P^{2}$, viewed as a collection of maps from $(\R P^{2},g_{0})$ to $(\R P^{2},g)$.  Because the energy of a map of a surface is bounded below by the area of its image, with equality precisely for conformal maps, we have $E_{2}(f) \geq A(\R P^{2},g)$ for $f \in \psi$.  However Pu's inequality (quoted in Theorem \ref{pu_theorem}) implies that, in the notation of Theorem \ref{rpn_p_energy_thm}, $A(\R P^{2},g) \geq \frac{2}{\pi} L^{\star^{2}}$, with equality only if $g$ has constant curvature.  Therefore, unless $g$ is isometric to a rescaling of $g_{0}$, the infimum of the energy in $\psi$ is strictly greater than the lower bound in Theorem \ref{rpn_p_energy_thm}. \\ 

In fact, the uniformization theorem implies that in any homotopy class of mappings from $(\R P^{2},g_{0})$ to a Riemannian manifold, the infimum of the energy is equal to the infimal area, cf. Lemma \ref{2-sphere_lemma}.  But although this determines the infimum of the energy in homotopy classes of mappings of $(\R P^{2},g_{0})$, Lemma \ref{rpn_area_lemma} suggests that for $n \geq 3$, the infimum of the energy in homotopy classes of mappings of $(\R P^{n},g_{0})$ is also usually greater than the lower bound in Theorem \ref{rpn_p_energy_thm} and may be more difficult to determine.  In Proposition \ref{rp3_inf_thm}, however, we will give upper and lower bounds for the infimum of the energy in a homotopy class of mappings of real projective $3$-space.  This result, together with a result of White which we quote in Theorem \ref{white_prop}, suggests that this infimum may be determined in part by the same geometric data as in our results for complex projective space.  We study this problem in more detail in \cite{Hois1}. 

\subsection*{Outline and Notation:} In Section \ref{stable_harmonic_mappings} we review some facts about harmonic maps and discuss some properties of pluriharmonic maps of complex projective space.  We also present a new proof of Theorem \ref{ohnita_theorem} and discuss its relationship to Ohnita's proof of this result in \cite{Oh1}, and to work of Burns, Burstall, de Bartolomeis and Rawnsley \cite{BBdBR1} and Lawson and Simons \cite{LS1}.  In Section \ref{lower_bounds}, we prove Theorems \ref{cpn_p_energy_thm} and \ref{rpn_p_energy_thm} and establish more properties of pluriharmonic maps of complex projective space.  In Section \ref{infima}, we prove Theorem \ref{infimum}.  We also discuss some corollaries of our results about pluriharmonic maps and some issues involved in finding the infimal energy in a homotopy class of mappings of real projective space.  Throughout, we write $\sigma(n)$ for the volume of the unit $n$-sphere, $J$ for the complex structure of any complex manifold, and $g_{0}$ for the canonical Riemannian metric on any compact rank-$1$ symmetric space---note that $(\C P^{N},g_{0})$ has volume $\frac{\pi^{N}}{N!}$ when normalized as in Theorem \ref{infimum}.  We write $E_{2}(F)$ for the energy, and $E_{p}(F)$ for the $p$-energy, of any map $F$ between Riemannian manifolds. 

\subsection*{Acknolwedgements:} I am happy to thank Werner Ballmann, Christopher Croke, and Joseph H.G. Fu for helpful conversations about this work, and the Max Planck Institute for Mathematics for support and hospitality.  

%%%%%%%%%%%%%%%%%%%%%%%%%%%%%%%%%%%%%%%%%%%%%%%%%%%%%%%%%%%%%%%%%%%%%%%%%%%%%%%%%%%%%%%%%%%%%%%%%%%%%%%%%%%%%%%%%%%%%%%%%

\section{Stable Harmonic and Pluriharmonic Mappings}  
\label{stable_harmonic_mappings} 

%%%%%%%%%%%%%%%%%%%%%%%%%%%%%%%%%%%%%%%%%%%%%%%%%%%%%%%%%%%%%%%%%%%%%%%%%%%%%%%%%%%%%%%%%%%%%%%%%%%%%%%%%%%%%%%%%%%%%%%%%

In this section we review some background about harmonic maps, and we discuss some properties of pluriharmonic maps of complex projectives space.  Because of Theorem \ref{ohnita_theorem}, we draw on these properties in discussing conditions under which a homotopy class of mappings of complex projective space contains an energy-minimizing map, and in proving Theorem \ref{cpn_p_energy_thm}.  We also present a new proof of Theorem \ref{ohnita_theorem}.  Like Ohnita's proof of Theorem \ref{ohnita_theorem} in \cite{Oh1}, our argument gives a slightly more general result, which we record in Propositions \ref{symmetric_space_prop} and \ref{cpn_pluriharmonic_prop}. \\ 

The energy of a Lipschitz map $F:(N^{n},h) \rightarrow (M^{m},g)$ of Riemannian manifolds is:  

\begin{equation}
\label{energy_eqn}
\displaystyle E_{2}(F) = \frac{1}{2} \int\limits_{N} |dF_{x}|^{2} dVol_{h}, 
\end{equation}
where $|dF_{x}|$ is the Euclidean norm of $dF:T_{x}N \rightarrow T_{F(x)}M$ at a point $x \in N$ at which $F$ is differentiable.  The Euler-Lagrange equation for maps which are critical for energy, called harmonic maps, can be formulated in terms of the second fundamental form of a map: 

\begin{definition}
\label{second_fund_form_defn}  

Let $F:(N,h) \rightarrow (M,g)$ be a smooth map of Riemannian manifolds, $\nabla^{h}$ and $\nabla^{g}$ the Levi-Civita connections of $(N,h)$ and $(M,g)$, and $F^{*}\nabla^{g}$ the induced connection in $F^{*}TM$.  The second fundamental form $\alpha_{F}$ of $F$ is the symmetric $F^{*}TM$-valued $2$-tensor on $N$ which, for vector fields $V$, $W$, satisfies $\alpha_{F}(V,W) = F^{*}\nabla^{g}_{V}F_{*}W - F_{*}(\nabla^{h}_{V}W)$. 
\end{definition} 

The section of $F^{*}TM$ given by taking the trace of the second fundamental form is called the tension field of the mapping $F$.  In \cite{ES1} Eells and Sampson showed that a smooth map of Riemannian manifolds is harmonic if and only if its tension field vanishes.  Continuous, weakly harmonic maps are smooth, cf. \cite[Ch.10]{Aub1}.  Continuous maps that minimizes energy in their homotopy class are therefore smooth, however a homotopy class may not contain an energy-minimizing map---for example, the infimum of the energy among maps homotopic to the identity mapping of the sphere $(S^{n},g_{0})$ of dimension $n \geq 3$ is $0$, but any such map $F$ is nonconstant and therefore has $E_{2}(F) > 0$.  A harmonic map is stable if the second variation of its energy is nonnegative.  Xin showed in \cite{Xi1} that for $n \geq 3$, the sphere $(S^{n},g_{0})$ does not admit any nonconstant stable harmonic map to any Riemannian manifold.  Ohnita showed in \cite{Oh3} that this is also the case for quaternionic projective space and the Cayley plane.  By way of a parallel with White's results in \cite{Wh1} and Theorem \ref{infimum}, which determine the infimum of the energy in all homotopy classes of mappings of simply-connected compact rank-$1$ symmetric spaces, the results of Xin and Ohnita, including Theorem \ref{ohnita_theorem}, determine all stable harmonic maps of these spaces. \\ 

A pluriharmonic map from a K\"ahler manifold to a Riemannian manifold is a map $F$ whose second fundamental form $\alpha_{F}$ satisfies the identity $\alpha_{F}(JV,JW) = - \alpha_{F}(V,W)$.  Pluriharmonic maps are harmonic.  Holomorphic and antiholomorphic maps between K\"ahler manifolds are pluriharmonic---more generally, pluriharmonicity and holomorphicity are related as follows: 

\begin{lemma}
\label{pluri_characterization_lemma}
For a smoooth map $F:(X,h)\rightarrow(M,g)$ from a K\"ahler manifold to a Riemannian manifold, the following are equivalent:  

\begin{itemize}
\item[\textbf{i.)}] $F$ is pluriharmonic. 
\item[\textbf{ii.)}]\label{pcliii} For any germ of a complex curve $\Sigma \subseteq X$, $F|_{\Sigma}$ is harmonic. 
\item[\textbf{iii.)}]\label{pclii} For any K\"ahler manifold $(Y,\widetilde{h})$ and holomorophic map $G:(Y,\widetilde{h}) \rightarrow (X,h)$, $F \circ G$ is pluriharmonic.  In particular, $F$ is pluriharmonic as a map $(X,\widetilde{h})\rightarrow(M,g)$ for any K\"ahler metric $\widetilde{h}$ on $X$.  
\end{itemize}
\end{lemma}

\begin{proof} That iii.) implies i.) and ii.) is immediate.  That ii.) implies i.) is a theorem of Rawnsley, cf. \cite[Section 4]{BBdBR1}.  To see that that i.) implies iii.), we calculate the second fundamental form of $F \circ G$: letting $\widetilde{F}_{*}:G^{*}TX \rightarrow (F \circ G)^{*}TM$ be the bundle homomorphism induced by $F_{*}:TX \rightarrow F^{*}TM$, we have $\alpha_{F \circ G}(V,W) = \alpha_{\widetilde{F}}(G_{*}V,G_{*}W) + \widetilde{F}_{*}\alpha_{G}(V,W)$, which implies that $\alpha_{F \circ G}(JV,JW) = -\alpha_{F \circ G}(V,W)$. \end{proof}

Pluriharmonic maps of complex projective space retain some properties of holomorphic maps, even when the target is only assumed to be a Riemannian manifold.  For the Riemann sphere, we have the following result of Lemaire:  

\begin{theorem}{\em (Lemaire \cite[Theorem 2.8]{Lem1})}
\label{harmonic_conf_thm}	
A harmonic map $f:(S^{2},g_{0}) \rightarrow (M,g)$ from the $2$-sphere to a Riemannian manifold is a conformal branched immersion. 
\end{theorem} 

Lemma \ref{pluri_characterization_lemma} implies an extension of Theorem \ref{harmonic_conf_thm} to pluriharmonic maps of higher-dimensional complex projective spaces, which we record in Lemma \ref{pluri_holom_prop}.  Later we quote a result of Ohnita, in Theorem \ref{bochner_result}, which subsumes this result, but we present a short proof of Lemma \ref{pluri_holom_prop} which introduces a definition and notation used in the proofs of Theorems \ref{infimum} and \ref{cpn_p_energy_thm}: 

\begin{lemma}
\label{pluri_holom_prop}
Let $F:(\C P^{N},g_{0}) \rightarrow (M,g)$ be a pluriharmonic map from complex projective space to a Riemannian manifold.  Then $F^{*}g$ is a Hermitian bilinear form on $\C P^{N}$; that is:  

\begin{equation}
\displaystyle F^{*}g(JV,JW) = F^{*}g(V,W). 
\end{equation}
\end{lemma}

\begin{proof} By Lemma \ref{pluri_characterization_lemma}, for all degree-$1$ curves $\C P^{1} \subseteq \C P^{N}$, the map $F|_{\C P^{1}}$ is harmonic.  Every unit tangent vector $\vec{u}$ to $\C P^{N}$ is tangent to a unique such curve, which we will denote $\mathsf{T}(\vec{u})$, and to which $J\vec{u}$ is also tangent.  By Theorem \ref{harmonic_conf_thm} as applied to $F|_{\mathsf{T}(\vec{u})}$, we have $|dF(\vec{u})| = |dF(J\vec{u})|$.  The polarization identity then implies the bilinear form $F^{*}g$ is Hermitian. \end{proof} 

In fact, we have: 

\begin{theorem}{\em (Ohnita \cite[Proposition 4.2]{Oh1})}
\label{bochner_result}
Let $(X,h)$ be a closed K\"ahler-Einstein manifold with positive Ricci curvature and $F:(X,h) \rightarrow (M,g)$ a pluriharmonic map to a Riemannian manifold.  Then $F^{*}g$ is a Hermitian bilinear form on $(X,h)$, as in Lemma \ref{pluri_holom_prop}.  Moreover, the $2$-form $\omega^{*}$ on $X$ given by $\omega^{*}(V,W) = F^{*}g(JV,W)$ is closed. 
\end{theorem}

We finish this section by presenting a new proof of Theorem \ref{ohnita_theorem}.  Like Ohnita's proof of Theorem \ref{ohnita_theorem} in \cite{Oh1}, and related work of Burns, Burstall, de Bartolomeis and Rawnsley in \cite{BBdBR1} which we discuss below, our proof is based on work of Lawson and Simons in \cite{LS1} and has the following outline: we first show that if $F$ is a stable harmonic map of complex projective space, the kernel of its second variation operator contains all gradients of first eigenfunctions---this is a special case of Proposition \ref{symmetric_space_prop} below.  We then show in Proposition \ref{cpn_pluriharmonic_prop} that harmonic maps of complex projective space with this property are pluriharmonic.  However we give new arguments for both steps in this proof, based on the result of Lemma \ref{Jacobi_lemma} below.  Lemma \ref{Jacobi_lemma} follows from a special case of the second variation formula for harmonic maps that appears in Ohnita's work:

\begin{lemma}{\em (Ohnita \cite[p.563--564]{Oh1})}    
\label{Ohnita_lemma}
Let $(N,h)$ and $(M,g)$ be Riemannian manifolds, with $(N,h)$ compact, and let $F:(N,h) \rightarrow (M,g)$ be a harmonic map.  Let $\mathcal{J}_{F}$ be the Jacobi operator of $F$; that is, for a variation $F_{t}$ of $F = F_{0}$ with $\frac{\partial F}{\partial t}|_{t=0} = \mathcal{W}$, 

\begin{equation}
\label{second_var_eqn}
\displaystyle \frac{d^{2}}{dt^{2}}(E_{2}(F_{t}))|_{t=0} = \int\limits_{N} g(\mathcal{J}_{F}(\mathcal{W}),\mathcal{W}) \ dVol_{N}. 
\end{equation}

Then for a vector field $V$ on $N$, letting $Tr(\alpha_{F} \circ \nabla^{h}V) = \sum\limits_{i=1}^{n} \alpha_{F}(\nabla^{h}_{e_{i}}V,e_{i})$, where $e_{1}, \dots, e_{n}$ is an orthonormal frame for $TN$, 

\begin{equation}
\label{Ohnita_lemma_eqn}
\displaystyle \mathcal{J}_{F}(F_{*}V) = -F_{*}\left(Tr(\nabla^{h}\nabla^{h}V) + Ric^{N}(V)\right) - 2 Tr(\alpha_{F} \circ \nabla^{h}V). 
\end{equation}
\end{lemma} 

For harmonic maps of K\"ahler manifolds, Lemma \ref{Ohnita_lemma} implies: 

\begin{lemma}
\label{Jacobi_lemma}
Let $(X,h)$ be a compact K\"ahler manifold and $F:(X,h) \rightarrow (M,g)$ a harmonic map to a Riemannian manifold.  Let $\mathcal{J}_{F}$ be the Jacobi operator of $F$ as in Lemma \ref{Ohnita_lemma}, and let $Tr(\alpha_{F}\circ\nabla^{h}V)$ be as in Lemma \ref{Ohnita_lemma}.  If $V$ is a holomorphic vector field on $X$, 

\begin{equation}
\label{Jacobi_lemma_eqn}
\displaystyle \mathcal{J}_{F}(F_{*}V) = -2 Tr(\alpha_{F}\circ\nabla^{h}V). 
\end{equation}
\end{lemma}

\begin{proof} Let $R(Y,W)U = \nabla^{h}_{Y}\nabla^{h}_{W}U - \nabla^{h}_{W}\nabla^{h}_{Y}U - \nabla^{h}_{[Y,W]}U$ be the curvature tensor of $(X,h)$ and $V$ a holomorphic vector field on $(X,h)$.  By extending a unit tangent vector $e \in T_{x}X$ to a locally defined holomorphic vector field $E$, we have:  
	
\begin{equation*}
\displaystyle \nabla^{h}\nabla^{h}V(e,e) + \nabla^{h}\nabla^{h}V(Je,Je) + R(V,e)e + R(V,Je)Je \smallskip
\end{equation*}
\begin{equation}
\displaystyle = \nabla^{h}_{E}[E,V] + \nabla^{h}_{JE}[JE,V] - \nabla^{h}_{[V,E]}E - \nabla^{h}_{[V,JE]}JE = 0, \smallskip 
\end{equation}
which implies $Tr(\nabla^{h}\nabla^{h}V) + Ric(V) = 0$.  By Lemma \ref{Ohnita_lemma} this implies the result. \end{proof} 

Lemma \ref{Jacobi_lemma} implies a property of harmonic maps of Hermitian symmetric spaces which also follows from Ohnita's work in \cite[Section 2]{Oh1}, and from work of Burns, Burstall, de Barolomeis and Rawnsley \cite[Lemma 2.3]{BBdBR1}:   

\begin{proposition}
\label{symmetric_space_prop}
Let $F:(Z_{0},h_{0}) \rightarrow (M,g)$ be a harmonic map from a compact Hermitian symmetric space to a Riemannian manifold, $\mathcal{J}_{F}$ the Jacobi operator of $F$ as in Lemma \ref{Jacobi_lemma}, and $\mathfrak{g}$ the Lie algebra of Killing vector fields on $(Z_{0},h_{0})$.  For $V \in \mathfrak{g}$, let $\widetilde{V}$ be the holomorphic vector field $JV$ on $(Z_{0},h_{0})$, and let $II$ be the bilinear form on $\mathfrak{g}$ given by:  

\begin{equation}
\label{sym_space_prop_eqn}
\displaystyle II(V,W) = \int\limits_{Z_{0}} g(\mathcal{J}_{F}(F_{*}\widetilde{V}),F_{*}\widetilde{W}) dVol_{h_{0}}. 
\end{equation} 

Then $Tr(II) = 0$.  If $F$ is stable, then for any $V \in \mathfrak{g}$ we have $F_{*}\widetilde{V} \in \ker(\mathcal{J}_{F})$. 
\end{proposition}

\begin{proof} Letting $V_{1}, \dots, V_{r}$ be a basis for $\mathfrak{g}$ which is orthonormal relative to the inner product given by the negative of the Killing form, and letting $\widetilde{V}_{i} = JV_{i}$ as above, 
	
\begin{equation}
\label{sym_space_pf_eqn_1}
\displaystyle Tr(II) = \int\limits_{Z_{0}} \sum\limits_{i = 1}^{r} g(\mathcal{J}_{F}(F_{*}\widetilde{V}_{i}),F_{*}\widetilde{V}_{i}) dVol_{h_{0}}. 
\end{equation}

The pointwise value of the integrand $\sum_{i = 1}^{r} g(\mathcal{J}_{F}(F_{*}\widetilde{V}_{i}),F_{*}\widetilde{V}_{i})$ in (\ref{sym_space_pf_eqn_1}) is independent of the orthonormal basis $V_{i}$ for $\mathfrak{g}$.  For each $x \in Z_{0}$ there is an orthogonal decomposition $\mathfrak{p}_{x} \oplus \mathfrak{k}_{x}$ of $\mathfrak{g}$, where $\mathfrak{p}_{x}$ is the space of $V \in \mathfrak{g}$ with $\nabla^{h_{0}}V = 0$ at $x$ and $\mathfrak{k}_{x}$ is the space of $V \in \mathfrak{g}$ with $V = 0$ at $x$.  Choosing orthonormal bases $V_{1}, \dots, V_{n}$ for $\mathfrak{p}_{x}$ and $V_{n+1}, \dots, V_{r}$ for $\mathfrak{k}_{x}$ gives an orthonormal basis $V_{1}, \dots, V_{r}$ for $\mathfrak{g}$ for which, by Lemma \ref{Jacobi_lemma}, we have:  

\begin{equation}
\label{sym_space_pf_eqn_2}	
\displaystyle g(\mathcal{J}_{F}(F_{*}\widetilde{V}_{i}),F_{*}\widetilde{V}_{i}) = -2 g(Tr(\alpha_{F} \circ \nabla^{h_{0}}\widetilde{V}_{i}),F_{*}\widetilde{V}_{i}) = 0, \smallskip 
\end{equation}
because $\nabla^{h_{0}}\widetilde{V}_{i} = 0$ for $i = 1,\dots,n$ and $F_{*}\widetilde{V}_{i} = 0$ for $i = n+1,\dots,r$ at $x$.  The integrand in (\ref{sym_space_pf_eqn_1}) therefore vanishes identically and $Tr(II) = 0$. \\ 

If $F$ is stable, let $0 = \lambda_{0} < \lambda_{1} < \cdots < \lambda_{j} < \cdots$ be the distinct eigenvalues of $\mathcal{J}_{F}$ acting on sections of $F^{*}TM$, let $V_{1}, \dots, V_{r}$ be an orthonormal basis for $\mathfrak{g}$, and let $\widetilde{\mathcal{V}}_{i}^{j}$ the component of $F_{*}\widetilde{V}_{i}$ belonging to the $j^{th}$ eigenspace of $\mathcal{J}_{F}$.  We then have:   

\begin{equation}
\displaystyle 0 = Tr(II) = \sum\limits_{i=1}^{r} \int\limits_{Z_{0}} g(\mathcal{J}_{F}(F_{*}\widetilde{V}_{i}),F_{*}\widetilde{V}_{i}) = \sum\limits_{j = 0}^{\infty} \lambda_{j} \left( \sum\limits_{i = 1}^{r} \int\limits_{Z_{0}} |\widetilde{\mathcal{V}}_{i}^{j}|^{2} dVol_{h_{0}} \right), 
\end{equation} 
which implies that $\widetilde{\mathcal{V}}_{i}^{j} = 0$ for all $j \geq 1$ and $i=1,\dots,r$, and therefore that $F_{*}\widetilde{V}_{i}$ is in the $0$-eigenspace of $\mathcal{J}_{F}$. \end{proof}

Proposition \ref{symmetric_space_prop} and another application of Lemma \ref{Jacobi_lemma} imply the result of Theorem \ref{ohnita_theorem}: 

\begin{proposition}
\label{cpn_pluriharmonic_prop}

Let $F:(\C P^{N},g_{0}) \rightarrow (M,g)$ be a harmonic map from complex projective space to a Riemannian manifold and $\mathcal{J}_{F}$ its Jacobi operator.  Suppose that for all Killing vector fields $V$ on $(\C P^{N},g_{0})$, letting $\widetilde{V} = JV$, we have $F_{*}\widetilde{V} \in \ker(\mathcal{J}_{F})$.  Then $F$ is pluriharmonic.  In particular, stable harmonic maps of $(\C P^{N},g_{0})$ are pluriharmonic. 
\end{proposition}

\begin{proof} By Lemma \ref{Jacobi_lemma}, for all Killing vector fields $V$ on $(\C P^{N},g_{0})$, letting $\widetilde{V} = JV$, and letting $Tr(\alpha_{F}\circ\nabla^{g_{0}}\widetilde{V})$ be as in Lemma \ref{Ohnita_lemma}, we have:  

\begin{equation}
\label{cpn_pluri_pf_eqn_1}
\displaystyle -2 Tr(\alpha_{F}\circ\nabla^{g_{0}}\widetilde{V}) = \mathcal{J}_{F}(F_{*}\widetilde{V}) = 0. \smallskip 
\end{equation} 

At $x \in \C P^{N}$, the Lie algebra $\mathfrak{k}_{x}$ of Killing vector fields which vanish at $x$ is isomorphic to $\mathfrak{u}(N)$, and the identification $V \mapsto \nabla^{g_{0}}V$ gives an isomorphism of $\mathfrak{k}_{x}$ with the algebra of skew-Hermitian linear transformations of $T_{x}\C P^{N}$.  For any unit tangent vector $e$ to $\C P^{N}$ at $x$, there is therefore a Killing vector field $V \in \mathfrak{k}_{x}$ with $\nabla^{g_{0}}_{e}V = Je$, $\nabla^{g_{0}}_{Je}V = -e$ and $\nabla^{g_{0}}_{e'}V = 0$ for $e'$ orthogonal to $\lbrace e, Je \rbrace$.  By (\ref{cpn_pluri_pf_eqn_1}), 

\begin{equation*}
\displaystyle 0 = Tr(\alpha_{F}\circ\nabla^{g_{0}}\widetilde{V}) = \alpha_{F}(\nabla^{g_{0}}_{e}\widetilde{V},e) + \alpha_{F}(\nabla^{g_{0}}_{Je}\widetilde{V},Je) \smallskip 
\end{equation*} 
\begin{equation}
\displaystyle = -\left[ \alpha_{F}(e,e) + \alpha_{F}(Je,Je) \right], \smallskip 
\end{equation}
so that $\alpha_{F}(Je,Je) = -\alpha_{F}(e,e)$.  By the polarization identity, the bilinear form $\alpha_{F}$ then satisfies $\alpha_{F}(JV,JW) = -\alpha_{F}(V,W)$ and $F$ is pluriharmonic.  Proposition \ref{symmetric_space_prop} implies this is the case for stable harmonic maps of $(\C P^{N},g_{0})$. \end{proof} 

%%%%%%%%%%%%%%%%%%%%%%%%%%%%%%%%%%%%%%%%%%%%%%%%%%%%%%%%%%%%%%%%%%%%%%%%%%%%%%%%%%%%%%%%%%%%%%%%%%%%%%%%%%%%%%%%%%%%%%%%%

\section{Lower Bounds for Energy Functionals of Mappings}  
\label{lower_bounds} 

%%%%%%%%%%%%%%%%%%%%%%%%%%%%%%%%%%%%%%%%%%%%%%%%%%%%%%%%%%%%%%%%%%%%%%%%%%%%%%%%%%%%%%%%%%%%%%%%%%%%%%%%%%%%%%%%%%%%%%%%%

In this section, we prove Theorems \ref{cpn_p_energy_thm} and \ref{rpn_p_energy_thm}, and we establish more properties of pluriharmonic maps of complex projective space, in Proposition \ref{holomorphic_corollary} and Lemmas \ref{immersion_lemma} and \ref{pluri_homothety_lemma}.  The starting point for the proofs of these results is a formula for the energy of a map due to Croke, which we quote in Lemma \ref{chris_lemma}.  If $F:(N^{n},h) \rightarrow (M^{m},g)$ is a Lipschitz map of Riemannian manifolds and $x \in N$ a point at which $F$ is differentiable, $F^{*}g$ is a positive semidefinite symmetric bilinear form on $T_{x}N$ which can be diagonalized relative to $h$.  Letting $e_{1}, e_{2}, \cdots, e_{n}$ be an orthonormal basis for $T_{x}N$ (relative to $h$) of eigenvectors for $F^{*}g$, we have: 

\begin{equation}
\label{energy_frame_eqn}
\displaystyle |dF_{x}|^{2} = \sum\limits_{i=1}^{n} |dF(e_{i})|^{2}. 
\end{equation}

Note that if $F$ is a pluriharmonic map of $(\C P^{N},g_{0})$, or any closed K\"ahler-Einstein manifold with positive Ricci curvature, Theorem \ref{bochner_result} implies $F^{*}g$ can be diagonalized as in (\ref{energy_frame_eqn}) by a unitary basis $e_{1}, e_{2} = Je_{1}, \cdots, e_{2N} = Je_{2N-1}$, with $|dF(e_{i})| = |dF(Je_{i})|$.  For any map $F$ of Riemannian manifolds (\ref{energy_frame_eqn}) implies that, letting $U_{x}(N,h)$ be the unit tangent fibre of $(N,h)$, 

\begin{equation}
\label{norm_squared_formula}
\displaystyle |dF_x|^{2} = \frac{n}{\sigma(n-1)} \int\limits_{U_{x}(N,h)} |dF(\vec{u})|^{2} d\vec{u}. 
\end{equation}

Integrating the identity in (\ref{norm_squared_formula}) over $x \in (N,h)$ gives: 

\begin{lemma}{\em (Croke \cite{Cr1})}
\label{chris_lemma} 
For a Lipschitz map $F:(N^{n},h) \rightarrow (M^{m},g)$ of Riemannian manifolds, with $U(N,h)$ the unit tangent bundle of $(N,h)$, 

\begin{equation}
\displaystyle E_{2}(F) = \frac{n}{2 \sigma(n-1)} \int\limits_{U(N,h)} |dF(\vec{u})|^{2} d\vec{u}. 
\end{equation}
\end{lemma}

In the proofs of Theorems \ref{infimum}, \ref{cpn_p_energy_thm} and \ref{rpn_p_energy_thm}, we work with the following measure spaces associated to the canonical Riemannian metrics on real and complex projective space:  

\begin{definition}
\label{measure_space}
\begin{flushleft}
{\em \textbf{A.}} Let $\mathcal{G}$ be the space of oriented geodesics $\gamma$ in $(\R P^{n},g_{0})$; that is, the quotient of the unit tangent bundle $U(\R P^{n},g_{0})$ by the geodesic flow.  Let $\zeta : U(\R P^{n},g_{0}) \rightarrow \mathcal{G}$ be the quotient map, $dVol_{U}$ the canonical measure on $U(\R P^{n},g_{0})$, and $d\gamma$ the measure $\frac{1}{\pi} \zeta_{\#}dVol_{U}$ on $\mathcal{G}$. \smallskip 

{\em \textbf{B.}} Let $\mathcal{L}$ be the space of linearly embedded $\C P^{1} \subseteq \C P^{N}$; that is, the quotient of the unit tangent bundle $U(\C P^{N},g_{0})$ by the map $\mathsf{T}$ which sends $\vec{u} \in U(\C P^{N},g_{0})$ to the unique degree-$1$ curve $\mathsf{T}(\vec{u}) \cong \C P^{1}$ to which $\vec{u}$ is tangent, as in the proof of Lemma \ref{pluri_holom_prop}.  Let $dVol_{U}$ be the canonical measure on $U(\C P^{N},g_{0})$ and $d\mathcal{P}$ the measure $\frac{1}{2\pi^{2}}\mathsf{T}_{\#}dVol_{U}$ on $\mathcal{L}$. 
\end{flushleft}
\end{definition}

The total volumes of $\mathcal{G}$ and $\mathcal{L}$ in the measures $d\gamma$ and $d\mathcal{P}$ in Definition \ref{measure_space} are $\frac{\sigma(n)\sigma(n-1)}{2\pi}$ and $\frac{\pi^{2N-2}}{N!(N-1)!}$ respectively.  These normalizations are chosen so that, for example, if $\eta$ is an integrable function on $U(\R P^{n},g_{0})$, then by Fubini, 

\begin{equation}
\label{examp_calc}
\displaystyle \int\limits_{U(\R P^{n},g_{0})} \eta(\vec{u}) \ d\vec{u} = \int\limits_{\mathcal{G}} \int\limits_{\gamma} \eta(\gamma'(t)) dt d\gamma. 
\end{equation}

Lemma \ref{chris_lemma}, together with Fubini's theorem as in (\ref{examp_calc}), gives a result about maps of complex projective space due to Croke.  For completeness, we include the proof of this theorem: 

\begin{theorem}{\em (Croke \cite[Theorem 3]{Cr1})}
\label{cpn_energy_formula_lemma}
Let $F:(\C P^{N},g_{0}) \rightarrow (M,g)$ be a Lipschitz map from complex projective space to a Riemannian manifold, let $\mathcal{L}$ and $d\mathcal{P}$ be as in Definition \ref{measure_space}.B, and for $\mathcal{P} \in \mathcal{L}$, let $F|_{\mathcal{P}}$ be the map given by composing the inclusion $\mathcal{P} \subseteq \C P^{N}$ with $F$.  Then: 

\begin{equation}
\label{cpn_energy_formula_prop_eqn}
\displaystyle E_{2}(F) = \frac{N!}{\pi^{N-1}} \int\limits_{\mathcal{L}} E_{2}(F|_{\mathcal{P}}) d\mathcal{P}.  
\end{equation}
\end{theorem}

\begin{proof} For $\mathcal{P} \in \mathcal{L}$, let $U(\mathcal{P},g_{0})$ be the unit tangent bundle of $\mathcal{P}$ in the metric $g_{0}|_{\mathcal{P}}$.  Lemma \ref{chris_lemma} and Fubini's theorem as in (\ref{examp_calc}) imply: 
	
\begin{equation}
\label{cpn_energy_lemma_pf_eqn}
\displaystyle E_{2}(F) = \frac{N}{\sigma(2N-1)} \int\limits_{\mathcal{L}} \int\limits_{U(\mathcal{P},g_{0})} |dF(\vec{u})|^{2} \ d\vec{u} \ d\mathcal{P}. 
\end{equation}

For all $\mathcal{P} \in \mathcal{L}$, the energy of $F|_{\mathcal{P}}$ can be calculated via Lemma \ref{chris_lemma}, which gives the result. \end{proof} 

For pluriharmonic maps of complex projective space, we then have: 

\begin{proposition}
\label{holomorphic_corollary}
Let $F:(\C P^{N},g_{0}) \rightarrow (M,g)$ be a pluriharmonic map to a Riemannian manifold, and let $\mathcal{L}$ be as in Definition \ref{measure_space}.B.  Then there is a constant $\mathcal{A}_{*}$ such that for all $\mathcal{P} \in \mathcal{L}$, we have $E_{2}(F|_{\mathcal{P}}) = |F(\mathcal{P})| = \mathcal{A}_{*}$, and, letting $C_{N}$ be as in Theorem \ref{infimum}, we have $E_{2}(F) = C_{N} \mathcal{A}_{*}$.  In particular, if $F$ is a holomorphic or antiholomorphic map to a compact K\"ahler manifold, we have $E_{2}(F) = C_{N} A^{\star}$, where $A^{\star}$ is as in Theorem \ref{infimum}. 
\end{proposition}

\begin{proof} Because $F$ is pluriharmonic, Lemma \ref{pluri_characterization_lemma} implies $F|_{\mathcal{P}}$ is harmonic for all $\mathcal{P}  \in \mathcal{L}$.  Theorem \ref{harmonic_conf_thm} then implies $F|_{\mathcal{P}}$ is conformal, so $E_{2}(F|_{\mathcal{P}})$ is equal to the area of its image.  To see that $E_{2}(F|_{\mathcal{P}})$ is the same for all $\mathcal{P} \in \mathcal{L}$, note that any two elements $\mathcal{P}_{0}$, $\mathcal{P}_{1}$ of $\mathcal{L}$ can be joined by a $1$-parameter family $\lbrace \mathcal{P}_{t} \rbrace_{0 \leq t \leq 1}$, by composing the inclusion $\mathcal{P}_{0} \subseteq \C P^{N}$ with a $1$-parameter family of isometries of $(\C P^{N},g_{0})$.  Because $F|_{\mathcal{P}_{t}}$ is harmonic for all $0 \leq t \leq 1$, the energy of $F|_{\mathcal{P}_{t}}$ is constant in $t$, so $E_{2}(F|_{\mathcal{P}_{0}}) = E_{2}(F|_{\mathcal{P}_{1}})$.  Theorem \ref{cpn_energy_formula_lemma} then implies $E_{2}(F) = C_{N}\mathcal{A}_{*}$, where $\mathcal{A}_{*}$ is the common value of $E_{2}(F|_{\mathcal{P}})$ for $\mathcal{P} \in \mathcal{L}$.  If $F$ is a holomorphic or antiholomorphic map to a compact K\"ahler manifold $(X,h)$, then $F(\mathcal{P})$ is a complex curve in $(X,h)$ for all $\mathcal{P} \in \mathcal{L}$ and therefore has minimal area in its homology class, so that $\mathcal{A}_{*} = A^{\star}$. \end{proof} 

Burns, Burstall, de Bartolomeis and Rawnsley showed in \cite[Lemma 6]{BBdBR1} that nonconstant pluriharmonic maps of compact simple Hermitian symmetric spaces are immersions on nonempty open subsets.  We have the following stronger statement for maps of complex projective space: 

\begin{lemma}
\label{immersion_lemma}
A nonconstant pluriharmonic map $F:(\C P^{N},g_{0}) \rightarrow (M,g)$ from complex projective space to a Riemannian manifold is an immersion on an open, dense subset $\Omega$ of $\C P^{N}$. 
\end{lemma}

\begin{proof} We have $rk(dF) = 2N$ on the subset $\Omega$ of $\C P^{N}$ where $\sqrt{\det(dF^{T} \circ dF)} \neq 0$, which is open.  To see that $\Omega$ is dense, suppose $x \in \C P^{N}$ is a point with $rk(dF_{x}) < 2N$.  By Theorem \ref{bochner_result}, $\ker(dF_{x})$ and $\ker(dF_{x})^{\perp}$ are complex subspaces of $T_{x}\C P^{N}$, so $rk(dF_{x}) = 2k = \dim(\ker(dF_{x})^{\perp})$.  Let $\vec{u} \in \ker(dF_{x})$, and let $\mathsf{T}(\vec{u}) \in \mathcal{L}$ be the degree-$1$ curve to which $\vec{u}$ is tangent, as in the proof of Lemma \ref{pluri_holom_prop}.  Lemma \ref{pluri_characterization_lemma} implies $F|_{\mathsf{T}(\vec{u})}$ is harmonic and Proposition \ref{holomorphic_corollary} implies it is nonconstant.  Theorem \ref{harmonic_conf_thm} then implies it is a conformal branched immersion.  The structure of $F|_{\mathsf{T}(\vec{u})}$ near $x$ is that of a branch point, and for $y \neq x$ in a sufficiently small neighborhood of $x$ in $\mathsf{T}(\vec{u})$, the rank of $dF|_{\mathsf{T}(\vec{u})}$ at $y$ is $2$ and the rank of $dF$ on $T_{y}\C P^{N}$ is therefore at least $2k+2$.  An open subset of $\C P^{N}$ containing $x$ must therefore contain points $y$ with $rk(dF_{y}) \geq 2k+2$.  Repeating this argument, $\Omega$ must meet all open subsets of $\C P^{N}$. \end{proof}

The energy $E_{2}(F)$ fits naturally into a $1$-parameter family of functionals: 

\begin{definition}[cf. \cite{Wh1,HL2,Wh3,We1}]
\label{p_energy_def}
Let $F:(N^{n},h) \rightarrow (M^{m},g)$ be a Lipschitz map of Riemannian manifolds.  For $p \geq 1$, the $p$-energy of $F$ is: 

\begin{equation}
\label{p-energy_eqn}
\displaystyle E_{p}(F) = \frac{1}{2} \int\limits_{N} |dF_{x}|^{p} dVol_{h}.  
\end{equation}
\end{definition}

Note that for $p \neq 2$, continuous maps which are critical for $E_{p}(F)$ need not be smooth, cf. \cite{HL2}.  Note also that the definition of $p$-energy in some papers differs from (\ref{p-energy_eqn}) by a constant multiple.  The following lower bound for the $p$-energy of a map generalizes the well-known fact that for mappings of surfaces, the energy is bounded below by the area of the image, with equality precisely for conformal maps: 

\begin{lemma}
\label{elementary_lemma}
Let $(N^{n},h)$ be a finite-volume Riemannian manifold and $F:(N^{n},h) \rightarrow (M^{m},g)$ a Lipschitz map, and define $Vol_{h}(N,F^{*}g)$ to be: 

\begin{equation}
\label{pullback_vol_eqn}
\displaystyle \int\limits_{N} \sqrt{det(dF_{x}^{T} \circ dF_{x})} dVol_{h}. 
\end{equation}

Then for $p \geq n$, 

\begin{equation}
\label{g_dim_energy_eqn}
\displaystyle E_{p}(F) \geq \frac{n^{\frac{p}{2}} Vol_{h}(N,F^{*}g)^{\frac{p}{n}}}{2 \ Vol(N,h)^{\frac{p-n}{n}}}. \smallskip 
\end{equation}

For $p=n$, equality holds if and only if $dF_{x}$ is a homothety at almost all $x \in N$.  For $p > n$, equality holds if and only if $dF_{x}$ is a homothety, by a constant factor $\kappa_{F}$, at almost all $x \in N$. \end{lemma}

\begin{proof} For $p=n$ this follows from the pointwise inequality $|dF_{x}|^{p} \geq n^{\frac{p}{2}} (\sqrt{det(dF_{x}^{T} \circ dF_{x})})^{\frac{p}{n}}$, which follows from (\ref{energy_frame_eqn}) and the arithmetic-geometric mean inequality for the eigenvalues of $F^{*}g$ relative to $h$.  For $p > n$, this follows from the result for $p = n$ and H\"older's inequality. \end{proof}

Note that in Lemma \ref{elementary_lemma}, equality for at least one $p > n$ implies equality for all $p \geq n$, and in fact that (\ref{g_dim_energy_eqn}) is an equality for all $p \geq 1$.  If $F$ is smooth, equality for $p=n$ in Lemma \ref{elementary_lemma} says that $F$ is semiconformal, i.e. $F^{*}g = \mu(x) h$ for a nonnegative function $\mu$ on $N$, and equality for $p > n$ says that $F$ is a homothety, i.e. $F^{*}g$ is a rescaling of $h$.  Lemma \ref{elementary_lemma} leads to the following rigidity statement for pluriharmonic maps of complex projective space: 

\begin{lemma}
\label{pluri_homothety_lemma}
Let $F:(\C P^{N},g_{0}) \rightarrow (M,g)$ be a pluriharmonic map from complex projective space to a Riemannian manifold with constant energy density $|dF|^{2}$.  Then $F$ is a homothety. 
\end{lemma}

\begin{proof} By Proposition \ref{holomorphic_corollary}, the energy of $F$ is equal to $C_{N}\mathcal{A}_{*}$, where $\mathcal{A}_{*}$ is the common value of $E_{2}(F|_{\mathcal{P}})$ for $\mathcal{P} \in \mathcal{L}$.  The constancy of $|dF|^{2}$ then implies that for all $p \geq 1$, 

\begin{equation}
\label{pluri_homothety_pf_eqn_1}
\displaystyle E_{p}(F) = \frac{\pi^{N}}{2 N!} \left( \frac{2N}{\pi} \mathcal{A}_{*} \right)^{\frac{p}{2}}. \smallskip 
\end{equation}

Letting $\omega^{*}$ be the closed $2$-form on $\C P^{N}$ associated to the Hermitian form $F^{*}g$ as in Theorem \ref{bochner_result}, for $\mathcal{P} \in \mathcal{L}$ we have: %% Let $\omega^{*}$ be the closed $2$-form on $\C P^{N}$ given by $\omega^{*}(V,W) = F^{*}g(JV,W)$, as in Theorem \ref{bochner_result}.  For $\mathcal{P} \in \mathcal{L}$ we have: 

\begin{equation}
\label{pluri_homothety_pf_eqn_2}
\displaystyle \int\limits_{\mathcal{P}}\omega^{*} = |F(\mathcal{P})| = \mathcal{A}_{*}. 
\end{equation}

Letting $\omega_{0}$ be the K\"ahler form of $g_{0}$, this implies that $\omega^{*}$ is cohomologous to $\frac{\mathcal{A}_{*}}{\pi} \omega_{0}$, and therefore that, letting $Vol_{g_{0}}(\C P^{N},F^{*}g)$ be as defined in Lemma \ref{elementary_lemma}, 

\begin{equation}
\label{pluri_homothety_pf_eqn_3}
\displaystyle Vol_{g_{0}}(\C P^{N},F^{*}g) = \frac{1}{N!} \int\limits_{\C P^{N}} \omega^{*^{N}} = \frac{\mathcal{A}_{*}^{N}}{N!}. 
\end{equation}

By (\ref{pluri_homothety_pf_eqn_1}) and (\ref{pluri_homothety_pf_eqn_3}), $E_{p}(F)$ is equal to the lower bound for $p$-energy in Lemma \ref{elementary_lemma} for all $p \geq 2N$.  Because this is the case for $p > 2N$, Lemma \ref{elementary_lemma} implies $F$ is a homothety. \end{proof}

We now prove Theorem \ref{cpn_p_energy_thm}.  In the proof of Theorem \ref{cpn_p_energy_thm} we also show, as Croke showed in \cite{Cr1}, that the value given in Theorem \ref{infimum} for the infimum of the energy in a homotopy class $\Phi$ from complex projective space to a Riemannian manifold is a lower bound for the energy of maps $F \in \Phi$.  We complete the proof of Theorem \ref{infimum} in Section \ref{infima} by showing that it is the infimum in each homotopy class. 

\begin{proof}[Proof of Theorem \ref{cpn_p_energy_thm}] Let $F:(\C P^{N},g_{0}) \rightarrow (M,g)$ be a Lipschitz map, and let $\mathcal{L}$ be as in Definition \ref{measure_space}.B.  By Theorem \ref{cpn_energy_formula_lemma} and the fact that $E_{2}(F|_{\mathcal{P}}) \geq |F(\mathcal{P})| \geq A^{\star}$ for all $\mathcal{P}  \in \mathcal{L}$, where $A^{\star}$ is the infimal area as above, 
		
\begin{equation}
\label{cpn_p_pf_eqn_1}
\displaystyle E_{2}(F) \geq C_{N} A^{\star}, \smallskip 
\end{equation}
where $C_{N}$ is as in Theorem \ref{infimum}.  For $p > 2$, H\"older's inequality then implies:  

\begin{equation}
\label{cpn_p_pf_eqn_2}
\displaystyle E_{p}(F) = \frac{1}{2} \int\limits_{\C P^{N}} |dF|^{p} dVol_{g_{0}} \geq \left( \frac{\pi^{N}}{2 N!} \right)^{1 - \frac{p}{2}} E_{2}(F)^{\frac{p}{2}}, 
\end{equation}
which together with (\ref{cpn_p_pf_eqn_1}) gives the inequality in Theorem \ref{cpn_p_energy_thm}. \\ 

Suppose $p > 2$ and equality holds for $E_{p}(F)$.  Equality then holds in (\ref{cpn_p_pf_eqn_1}), which implies $F$ is energy-minimizing in its homotopy class and pluriharmonic by Theorem \ref{ohnita_theorem}.  Equality also holds in H\"older's inequality in (\ref{cpn_p_pf_eqn_2}), which implies $|dF|^{2}$ is constant and $F$ is a homothety by Lemma \ref{pluri_homothety_lemma}.  The harmonicity of $F$ implies its image is a minimal submanifold. \end{proof}  

In Theorem \ref{rpn_p_energy_thm}, proven below, we establish a lower bound for the $p$-energy of a map of real projective space and then show that a map which attains this lower bound is totally geodesic; that is, its second fundamental form vanishes.  Totally geodesic maps of K\"ahler manifolds are pluriharmonic.  In this sense, the conclusion that a map of real projective space which attains the lower bound in Theorem \ref{rpn_p_energy_thm} is totally geodesic is stronger than the conclusion in Theorem \ref{cpn_p_energy_thm} that such a map of complex projective space is pluriharmonic.  The Veronese embedding $\C P^{N} \rightarrow \C P^{\binom{N+d}{d}-1}$ of degree $d \geq 2$ shows that a map of complex projective space which attains the lower bound for $p$-energy in Theorem \ref{cpn_p_energy_thm} need not be totally geodesic. 

\begin{proof}[Proof of Theorem \ref{rpn_p_energy_thm}] Let $F$ be a Lipschitz map from $(\R P^{n},g_{0})$ to a Riemannian manifold $(M,g)$, and let $\mathcal{G}$ and $d\gamma$ be as in Definition \ref{measure_space}.A.  For $\gamma \in \mathcal{G}$, let $\gamma: [0,\pi] \rightarrow \R P^{n}$ be a unit-speed parametrization of $\gamma$ and $F \circ \gamma: [0,\pi] \rightarrow M$ the associated parametrization of the path $F \circ \gamma$ in $(M,g)$.  Note that $F \circ \gamma$ is a Lipschitz path with a well-defined length $|F \circ \gamma|$ for all $\gamma \in \mathcal{G}$.  By the identity (\ref{norm_squared_formula}), the Cauchy-Schwarz inequality, and Fubini's theorem as in (\ref{examp_calc}),   

\begin{equation*}
\displaystyle \displaystyle E_{1}(F) = \frac{\sqrt{n}}{2\sqrt{\sigma(n-1)}} \int\limits_{\R P^{n}} \left( \int\limits_{U_{x}(\R P^{n},g_{0})} |dF(\vec{u})|^{2} d\vec{u} \right)^{\frac{1}{2}} dx \geq \frac{\sqrt{n}}{2 \sigma(n-1)} \int\limits_{U(\R P^{n},g_{0})} |dF(\vec{u})| d\vec{u}  \smallskip
\end{equation*}
\begin{equation}
\label{rpn_pf_eqn_1}
\displaystyle = \frac{\sqrt{n}}{2 \sigma(n-1)} \int\limits_{\mathcal{G}}\int\limits_{\gamma} |(F \circ \gamma)'(t)| dt d\gamma = \frac{\sqrt{n}}{2 \sigma(n-1)} \int\limits_{\mathcal{G}} |F \circ \gamma| d\gamma. 
\end{equation}

Because $|F \circ \gamma|$ is bounded below by the infimal length $L^{\star}$, this implies the inequality in Theorem \ref{rpn_p_energy_thm} for $p=1$.  For $p > 1$, H\"older's inequality implies: 

\begin{equation}
\label{rpn_pf_eqn_3}
\displaystyle E_{p}(F) = \frac{1}{2} \int\limits_{\R P^{n}} |dF_{x}|^{p} dVol_{g_{0}} \geq \left(\frac{\sigma(n)}{4}\right)^{1-p} E_{1}(F)^{p}, 
\end{equation}
which together with the inequality for $p=1$ gives the inequality for $p > 1$. \\ 

Suppose equality holds for $p=1$. \\ 

Supposing only that $F$ is Lipschitz, this implies equality holds in the Cauchy-Schwarz inequality in (\ref{rpn_pf_eqn_1}) for a.e. $x \in \R P^{n}$.  For all $x$ at which this equality holds, $|dF_{x}(\vec{u})|$ depends only on $x$.  This implies $F^{*}g$ is a.e. equal to $\mu(x) g_{0}$, where $\mu$ is a nonnegative function on $\R P^{n}$.  Equality also implies that $|F \circ \gamma| = L^{\star}$ for almost all $\gamma \in \mathcal{G}$.  Because $|F \circ \gamma| \geq L^{\star}$ and $|F \circ \gamma|$ is a lower semicontinuous function of $\gamma \in \mathcal{G}$, we have $|F \circ \gamma| = L^{\star}$ for all $\gamma$.  The image via $F$ of each $\gamma \in \mathcal{G}$ is therefore a closed geodesic in $(M,g)$, of minimal length $L^{\star}$ in its free homotopy class, although we do not know that $t \mapsto (F \circ \gamma)(t)$ is a constant-speed parametrization. \\ 

If in addition $F$ is a smooth immersion, then because each geodesic $\gamma \in \mathcal{G}$ maps to a closed geodesic in $(M,g)$, the image of $F$ is a totally geodesic submanifold of $(M,g)$.  Because each geodesic $\gamma$ maps to a curve $F \circ \gamma$ of minimal length in its free homotopy class in $(M,g)$, $F^{*}g$ is a Blaschke metric on $\R P^{n}$, cf. Remark \ref{blaschke_remark} below.  The Blaschke conjecture for real projective space, proven in dimension $2$ by Green \cite{Gn} and in all dimensions by Berger and Kazdan \cite{Be1}, implies $F^{*}g$ has constant curvature and is therefore isometric up to scale to $g_{0}$.  After rescaling if necessary, letting $\iota$ be a diffeomorphism of $\R P^{n}$ which gives an isometry from $(\R P^{n},F^{*}g)$ to $(\R P^{n},g_{0})$, we have $\iota^{*}g_{0} = F^{*}g = \mu(x) g_{0}$, where $\mu$ is the semiconformal factor above and is in fact a conformal factor, i.e. is everywhere-defined and positive, because $F$ is a smooth immersion.  This implies $\iota$ is a conformal diffeomorphism of $(\R P^{n},g_{0})$.  The classification of conformal diffeomorphisms of the covering $(S^{n},g_{0})$ of $(\R P^{n},g_{0})$ then  implies $\iota$ is an isometry, which implies $\mu$ is constant and $F$ is an isometry or a homothety. \\ 

Now suppose $p > 1$ and equality holds for $p$. \\ 

Supposing only that $F$ is Lipschitz, this implies all the conditions which hold for Lipshitz maps when equality holds for $p=1$ and also implies equality in H\"older's inequality in (\ref{rpn_pf_eqn_3}).  This implies $|dF|$ is a.e. equal to a constant $k_{F}$, and therefore that equality holds for all $p \geq 1$.  Because equality holds for $p = 2$, $F$ is harmonic and therefore smooth, and because $|dF|$ is constant, the semiconformal factor $\mu$ is a constant function.  $F$ is therefore either a constant map, if $\mu \equiv 0$, or a homothety if $\mu \equiv \text{const.} > 0$.  In the latter case $F$ is totally geodesic by the result for $p=1$. \end{proof} 

\begin{remark}
\label{blaschke_remark}
To see that equality when $p=1$ and $F$ is an immersion implies that $F^{*}g$ is a Blaschke metric on $\R P^{n}$, i.e. that the covering metric on $S^{n}$ has diameter equal to its injectivity radius, let $c:[0,L^{\star}] \rightarrow (\R P^{n},F^{*}g)$ be a unit-speed geodesic and $\widetilde{c}$ a lift of $c$ to $S^{n}$.  The geodesic $\widetilde{c}([0,L^{\star}])$ is minimizing in $S^{n}$, because $c([0,L^{\star}])$ is length-minimizing in its homotopy class in $(\R P^{n},F^{*}g)$, and $\widetilde{c}(L^{\star})$ is conjugate to $\widetilde{c}(0)$ along $\widetilde{c}$, because all geodesics based at $\widetilde{c}(0)$ of length $L^{\star}$ are lifts of homotopic closed curves in $\R P^{n}$ and therefore have the same endpoint.  Note that there are several equivalent properties that are sometimes taken as the definition of a Blaschke metric, discussed in \cite[Ch. 5]{Be1}. 
\end{remark} 

Also, note that the case $p=1$ of Theorem \ref{rpn_p_energy_thm} is false without the stipulation that $n \geq 2$: the $1$-energy of any diffeomorphism of $\R P^{1}$ is equal to that of the identity.  It would be interesting to find weaker hypotheses under which this part of the theorem holds.  One can modify the proof of Theorem \ref{rpn_p_energy_thm} above to show that if a map $F:(\R P^{n},g_{0}) \rightarrow (M,g)$ realizes the lower bound for $1$-energy in the theorem, then this is also the case for $F|_{\R P^{2}}$ for all totally geodesic $\R P^{2} \subseteq \R P^{n}$.  A proof of this part of Theorem \ref{rpn_p_energy_thm} for $(\R P^{2},g_{0})$ under weaker assumptions therefore gives a result in all dimensions.  In this way, for example, one can show via \cite[Theorem 4.53]{Be1} that the $p=1$ case of Theorem \ref{rpn_p_energy_thm} extends to $C^{1}$ immersions. 

%%%%%%%%%%%%%%%%%%%%%%%%%%%%%%%%%%%%%%%%%%%%%%%%%%%%%%%%%%%%%%%%%%%%%%%%%%%%%%%%%%%%%%%%%%%%%%%%%%%%%%%%%%%%%%%%%%%%%%%%%

\section{Infima of the Energy Functional in Homotopy Classes of Mappings}  
\label{infima} 

%%%%%%%%%%%%%%%%%%%%%%%%%%%%%%%%%%%%%%%%%%%%%%%%%%%%%%%%%%%%%%%%%%%%%%%%%%%%%%%%%%%%%%%%%%%%%%%%%%%%%%%%%%%%%%%%%%%%%%%%%

In this section we prove Theorem \ref{infimum}, and we discuss some corollaries of our results about pluriharmonic maps of complex projective space, in Remark \ref{pluriharmonic_remark}.  We also discuss the problem of determining the infimum of the energy in a homotopy class of mappings of real projective space and establish a two-sided estimate for the infimum of the energy in a homotopy class of mappings of real projective $3$-space, in Proposition \ref{rp3_inf_thm}. \\ 

The proof of Theorem \ref{infimum} begins with the fact that in any homotopy class of mappings from the $2$-sphere to a Riemannian manifold, the infimal energy is equal to the infimal area.  Sacks and Uhlenbeck established in \cite{SU} that the second homotopy group of any closed Riemannian manifold is generated by free homotopy classes containing harmonic maps that minimize both energy and area, however, as they discuss, the sum of two such classes may not contain a continuous, energy-minimizing map.  We show that the two infima are equal in all homotopy classes via the following lemma, adapted from the solution of Plateau's problem: 

\begin{lemma}
\label{2-sphere_lemma}

Let $f:(S^{2},g_{0}) \rightarrow (M,g)$ be a smooth map to a Riemannian manifold and $A(f)$ the area of its image.  Then for any $\delta > 0$ there is a diffeomorphism $\phi: S^{2} \rightarrow S^{2}$, homotopic to the identity, such that $E_{2}(f \circ \phi) < A(f) + \delta$.  In particular, $f$ is homotopic to a map with energy less than $A(f) + \delta$.  If $f$ is antipodally invariant, we can choose $\phi$ and the homotopy to be antipodally invariant, so that the same conclusions hold for mappings of $(\R P^{2},g_{0})$. \end{lemma}

\begin{proof} If $f$ is an immersion, the result follows immediately from the uniformization theorem. Otherwise, let $f_{r}:(S^{2},g_{0}) \rightarrow (M,g) \times (S^{2},rg_{0})$ be the product of $f$ with a homothety to a round sphere $(S^{2},rg_{0})$ of constant curvature $\frac{1}{r}$.  The uniformization theorem then implies there is a diffeomorphism $\phi_{r}:(S^{2},g_{0}) \rightarrow (S^{2},g_{0})$ homotopic to the identity such that $f_{r} \circ \phi_{r}$ is conformal.  By Lemma \ref{elementary_lemma}, we have $E_{2}(f_{r} \circ \phi_{r}) = |Im(f_{r} \circ \phi_{r})| = |Im(f_{r})|$.  An elementary calculation shows that $E_{2}(f \circ \phi_{r}) < E_{2}(f_{r} \circ \phi_{r})$, and that by choosing $r$ small enough, we can ensure that $|Im(f_{r})| < |Im(f)| + \delta$ for any $\delta > 0$.  Because $\phi_{r}$ is homotopic to the identity, $f$ is homotopic to $f \circ \phi_{r}$.  If $f$ is antipodally invariant, the uniformization theorem and homotopy lifting property imply that we can choose $\phi_{r}$ and the homotopy to be antipodally invariant. \end{proof}

\begin{proof}[Proof of Theorem \ref{infimum}] Let $\Phi$ be a homotopy class of mappings from $(\C P^{N},g_{0})$ to a Riemannian manifold $(M,g)$.  The inequality (\ref{cpn_p_pf_eqn_1}) in the proof of Theorem \ref{cpn_p_energy_thm} shows that the value given in Theorem \ref{infimum} is a lower bound for the energy of maps $F \in \Phi$.  To see that this lower bound is the infimum, let $[z_{0}:z_{1}:z_{2}:\dots:z_{N}]$ be homogeneous coordinates for $\C P^{N}$ and $\mathcal{P}_{0}$ the following degree-$1$ curve in $\C P^{N}$: 
		
\begin{equation}
\displaystyle \mathcal{P}_{0} = \lbrace [z_{0}:z_{1}:0:\dots:0] \ | \ (z_{0},z_{1}) \neq (0,0) \rbrace. \smallskip 
\end{equation}   

For $\lambda > 0$, let $T_{\lambda}:\C P^{N} \rightarrow \C P^{N}$ be the projective linear transformation associated to the following linear transformation of $\C^{N + 1}$:  

\begin{equation}
\label{inf_pf_eqn_2}
\displaystyle (z_{0},z_{1},z_{2},\dots,z_{N}) \mapsto (\lambda z_{0},\lambda z_{1},z_{2},\dots,z_{N}). \smallskip 
\end{equation}

Let $\mathcal{L}$ be as in Definition \ref{measure_space}.B.  Because $T_{\lambda}:\C P^{N} \rightarrow \C P^{N}$ is biholomorphic, $T_{\lambda}|_{\mathcal{P}}$ is conformal for all $\mathcal{P} \in \mathcal{L}$.  For any map $F:(\C P^{N},g_{0}) \rightarrow (M,g)$ and $\mathcal{P} \in \mathcal{L}$, we therefore have:  

\begin{equation}
\label{inf_pf_eqn_3}
\displaystyle E_{2}(F \circ T_{\lambda}|_{\mathcal{P}}) = E_{2}(F|_{T_{\lambda}(\mathcal{P})}). \smallskip 
\end{equation}

Let $\mathcal{C}(\mathcal{P}_{0})$ be the intersection of the cut loci in $(\C P^{N},g_{0})$ of points $x \in \mathcal{P}_{0}$; that is:   

\begin{equation}
\displaystyle \mathcal{C}(\mathcal{P}_{0}) = \lbrace [0:0:z_{2}:\dots:z_{N}] \ | \ (z_{2},\dots,z_{N}) \neq (0,\dots,0) \rbrace. \smallskip 
\end{equation}

For all $\mathcal{P} \in \mathcal{L}$ which do not intersect $\mathcal{C}(\mathcal{P}_{0})$, $T_{\lambda}(\mathcal{P})$ converges to $\mathcal{P}_{0}$ in the $C^{1}$ topology on submanifolds of $\C P^{N}$ as $\lambda \rightarrow \infty$.  By (\ref{inf_pf_eqn_3}), for any $C^{1}$ map $F$ and $\mathcal{P} \in \mathcal{L}$ disjoint from $\mathcal{C}(\mathcal{P}_{0})$, we therefore have:   

\begin{equation}
\label{inf_pf_eqn_4}
\displaystyle \lim\limits_{\lambda \rightarrow \infty} E_{2}(F \circ T_{\lambda}|_{\mathcal{P}})  = E_{2}(F|_{\mathcal{P}_{0}}). \smallskip 
\end{equation}

The set of $\mathcal{P} \in \mathcal{L}$ which intersect $\mathcal{C}(\mathcal{P}_{0})$ has measure $0$ relative to the measure $d\mathcal{P}$.  For any $C^{1}$ map $F:(\C P^{N},g_{0}) \rightarrow (M,g)$, Theorem \ref{cpn_energy_formula_lemma}, the limit in (\ref{inf_pf_eqn_4}), and the dominated convergence theorem then imply:  

\begin{equation}
\label{inf_pf_eqn_5}	
\displaystyle \lim\limits_{\lambda \rightarrow \infty} E_{2}(F \circ T_{\lambda}) = \lim\limits_{\lambda \rightarrow \infty} \frac{N!}{\pi^{N-1}} \int\limits_{\mathcal{L}} E_{2}(F \circ T_{\lambda}(\mathcal{P})) d\mathcal{P} = C_{N} E_{2}(F|_{\mathcal{P}_{0}}), 
\end{equation}
where $C_{N}$ is the constant in Theorem \ref{infimum}.  Letting $A^{\star}$ be the infimal area as above, Lemma \ref{2-sphere_lemma} and the homotopy extension property for $\C P^{1} \subseteq \C P^{N}$ imply that $\Phi$ contains maps $F$ with $E_{2}(F|_{\mathcal{P}_{0}}) < A^{\star} + \varepsilon$ for any $\varepsilon > 0$.  Letting $F$ be such a map in (\ref{inf_pf_eqn_5}) completes the proof. \end{proof} %%  of Theorem \ref{infimum}

\begin{remark}
\label{pluriharmonic_remark} 
For complex projective space, Theorem \ref{ohnita_theorem} and the results about pluriharmonic maps in Sections \ref{stable_harmonic_mappings} and \ref{lower_bounds} lead to a partial converse to Lichnerowicz's theorem in \cite{Li1} that holomorphic and antiholomorphic maps to K\"ahler manifolds minimize energy in their homotopy class, in the following two senses:  First, Theorems \ref{ohnita_theorem} and \ref{bochner_result} and Lemma \ref{immersion_lemma} imply that if $F:(\C P^{N},g_{0}) \rightarrow (M,g)$ is a nonconstant map which minimizes energy in its homotopy class, then $F^{*}g$ is a K\"ahler metric on an open, dense subset $\Omega$ of $\C P^{N}$ on which $F$ is an immersion.  The pluriharmonicity of $F$ also implies that the second fundamental form of $F(\Omega)$ in $(M,g)$ can be diagonalized by a unitary basis, like the second fundamental form of a K\"ahler submanifold.  In this sense, $F$ is indistinguishable along $\Omega$ ``to second order" from a holomorphic map to a K\"ahler manifold.  Second, letting $\omega^{*}$ be the closed $2$-form on $\C P^{N}$ associated to $F^{*}g$ as in Theorem \ref{bochner_result} and $\omega_{0}$ the K\"ahler form of $g_{0}$, $\omega^{*}$ is cohomologous up to scale to $\omega_{0}$, as in the proof of Lemma \ref{pluri_homothety_lemma}.  By the $\partial \overline{\partial}$-lemma, after rescaling if necessary, we have:

\begin{equation}
\displaystyle \omega^{*} = \omega_{0} + i\partial\overline{\partial} \xi, \smallskip 
\end{equation}
where $\xi$ is a smooth, real-valued function on $\C P^{N}$.  The positivity of the $2$-form $\omega_{0}$ and the nonnegativity of the $2$-form $\omega^{*}$ imply that for all $0 \leq \rho <1$, the $2$-form $\omega_{0} + \rho i\partial\overline{\partial} \xi$ is positive and thus induces a K\"ahler metric $g_{\rho}$ on $\C P^{N}$.  In this sense, $F$ can be seen as a limit of locally biholomorphic maps given by the identity mapping $(\C P^{N},g_{0}) \rightarrow (\C P^{N},g_{\rho})$ as $\rho \rightarrow 1$. \\ 

In some cases, stable harmonic and energy-miminizing maps of complex projective space must be holomorphic or antiholomorphic: Ohnita proved that stable harmonic maps between complex projective spaces are holomorphic or antiholomorphic in \cite{Oh1}.  Burns, Burstall, de Bartolomeis and Rawnsley then showed that this is the case for stable harmonic maps from complex projective space to any compact simple Hermitian symmetric space in \cite{BBdBR1}.  It also follows from the results of Lichnerowicz in \cite{Li1} that if a homotopy class of mappings between compact K\"ahler manifolds contains a holomorphic map, then all energy-minimizing maps in that class are holomorphic, with a corresponding statement for antiholomorphic maps.  In general however, although energy-minimizing maps of $(\C P^{N},g_{0})$ resemble holomorphic maps locally as described above, they may not be globally injective or orientation-preserving. 
\end{remark}

We conclude by discussing the problem of finding the infimum of the energy in a homotopy class of mappings of real projective space.  We will draw on the result of Pu's systolic inequality:  

\begin{theorem}{\em (Pu \cite{Pu}, see also \cite{CK1})}
\label{pu_theorem}
Let $g$ be a Riemannian metric on $\R P^{2}$, $A(\R P^{2},g)$ its area, and $\sys(g)$ the minimal length of a noncontractible curve in $(\R P^{2},g)$ (called a systole).  Then $A(\R P^{2},g) \geq \frac{2}{\pi} \sys(g)^{2}$.  Equality holds if and only if $g$ has constant curvature. 
\end{theorem}

We have the following formula for the energy of a map of real projective space: 

\begin{lemma}
\label{rpn_area_lemma}
Let $F:(\R P^{n},g_{0}) \rightarrow (M,g)$ be a Lipschitz map to a Riemannian manifold, $\mathcal{H}$ the set of totally geodesic $\mathcal{Q} \cong \R P^{2} \subseteq \R P^{n}$, and $d\mathcal{Q}$ the measure on $\mathcal{H}$ which is invariant under the action of the isometry group of $(\R P^{n},g_{0})$, normalized to have total volume $\frac{n \sigma(n)}{8 \pi}$.  Then:  

\begin{equation}
\displaystyle E_{2}(F) = \int\limits_{\mathcal{H}} E_{2}(F|_{\mathcal{Q}}) d\mathcal{Q}. 
\end{equation}
\end{lemma}

\begin{proof} Let $\mathcal{I}$ be the set of pairs $(\vec{u},\mathcal{Q})$, where $\vec{u} \in U(\R P^{n},g_{0})$, $\mathcal{Q} \in \mathcal{H}$, and $\vec{u}$ is tangent to $\mathcal{Q}$, and let $dv$ be the measure on $\mathcal{I}$ which is invariant under the isometry group of $(\R P^{n},g_{0})$, with total volume $\frac{n\sigma(n)}{4}$.  By Lemma \ref{chris_lemma}, applied to both $F$ and $F|_{\mathcal{Q}}$, %$E_{2}(F)$ is equal to: 
\begin{equation}
\displaystyle E_{2}(F) = \frac{n}{2\sigma(n-1)} \int\limits_{U(\R P^{n},g_{0})} |dF(\vec{u})|^{2} d\vec{u} = \int\limits_{\mathcal{I}} |dF(\vec{u})|^{2} dv = \int\limits_{\mathcal{H}} E_{2}(F|_{\mathcal{Q}}) d\mathcal{Q}. 
\end{equation}
\end{proof}

Let $\Psi$ be a homotopy class of mappings from $(\R P^{n},g_{0})$ to a Riemannian manifold $(M,g)$ and $\psi$ the class of mappings of $(\R P^{2},g_{0})$ represented by composing the inclusion $\R P^{2} \subseteq \R P^{n}$ with $F \in \Psi$.  Let $B^{\star}$ be the infimal area of mappings in $\psi$.  Then, as in the proof of Theorem \ref{cpn_p_energy_thm}, Lemma \ref{2-sphere_lemma} implies: 

\begin{equation}
\label{rpn_inf_ineq_1}
\displaystyle \inf\limits_{F \in \Psi} E_{2}(F) \geq \frac{n \sigma(n)}{8 \pi} B^{\star}. \smallskip 
\end{equation}

If the infimal area $B^{\star}$ is realized by an immersion $f:(\R P^{2},g_{0}) \rightarrow (M,g)$, then Theorem \ref{pu_theorem} implies: 

\begin{equation}
\label{pu_lemma_eqn}
\displaystyle B^{\star} \geq \frac{2}{\pi} \sys(f^{*}g)^{2} \geq \frac{2}{\pi} L^{\star^{2}}, \smallskip 
\end{equation}
where $L^{\star}$ is as in Theorem \ref{rpn_p_energy_thm}.  When the infimal area $B^{\star}$ is not realized by a smooth immersion, one can show that the inequality $B^{\star} \geq \frac{2}{\pi}L^{\star^{2}}$ in (\ref{pu_lemma_eqn}) holds by considering minimizing sequences for the area of maps $f \in \psi$.  Moreover, (\ref{pu_lemma_eqn}) must be a strict inequality unless the class $\psi$ contains a minimizing sequence with a fairly rigid characterization, particularly given the strengthened versions of Pu's inequality proven by Katz-Nowik and Katz-Sabourau in \cite{KN1,KS1}.  If the strict inequality $B^{\star} > \frac{2}{\pi}L^{\star^{2}}$ holds, then by (\ref{rpn_inf_ineq_1}) the infimal energy in the homotopy class $\Psi$ is strictly greater than the lower bound in Theorem \ref{rpn_p_energy_thm}. \\ 

Although Lemma \ref{rpn_area_lemma} suggests that the infimum of the energy in a homotopy class of mappings of $(\R P^{n},g_{0})$ is typically greater than the lower bound in Theorem \ref{rpn_p_energy_thm}, in Proposition \ref{rp3_inf_thm} we show that if $\Psi$ is a homotopy class of mappings of $(\R P^{3},g_{0})$ and $\psi$ is the induced class of mappings of $(\R P^{2},g_{0})$ as above, the infimum of the energy in $\Psi$ can be bounded from above as well as below by the infimal area in $\psi$.  An immediate corollary of the work of White in \cite{Wh1} implies that for all $n \geq 3$, the infimum of the energy in a homotopy class of mappings of $(\R P^{n},g_{0})$ is determined by the induced class of mappings of $(\R P^{2},g_{0})$ in the following sense: 

\begin{theorem}{\em (see \cite[Theorem 1]{Wh1})}
\label{white_prop}
Let $\Psi$ be a homotopy class of mappings from $(\R P^{n},g_{0})$ to a Riemannian manifold $(M,g)$, and let $\widehat{\Psi}$ be the union of all homotopy classes of mappings from $(\R P^{n},g_{0})$ to $(M,g)$ which give the same class of mappings of $(\R P^{2},g_{0})$ as $\Psi$ when composed with the inclusion $\R P^{2} \subseteq \R P^{n}$.  Then $\inf\limits_{F \in \Psi} E_{2}(F) = \inf\limits_{F \in \widehat{\Psi}} E_{2}(F)$. 
\end{theorem} 

\begin{remark}
\label{homotopy_dependence}
To derive Theorem \ref{white_prop} from \cite[Theorem 1]{Wh1}, note that in \cite{Wh1} maps are defined to be $2$-homotopic if their restrictions to the $2$-skeleton of a triangulation of their domain are homotopic.  If $F:(N,h) \rightarrow (M,g)$ is a map of a compact Riemannian manifold, it follows from \cite[Theorem 1]{Wh1} that the infimum of the energy in its homotopy class is equal to the infimum in its $2$-homotopy class.  Although $\R P^{2}$ is not the $2$-skeleton of a triangulation of $\R P^{n}$ (for $n \neq 2$), it is straightforward to check that the union of homotopy classes $\widehat{\Psi}$ in the theorem coincides with the $2$-homotopy class of maps $F \in \Psi$. 
\end{remark} 

In connection with Theorem \ref{infimum}, one can also infer from \cite[Theorem 1]{Wh1} that the infimum of the energy in a homotopy class of mappings from complex projective space to a Riemannian manifold depends only on the induced homomorphism on the second homotopy group.  These observations about the homotopy-theoretic dependence of the infimum do not seem to lead to explicit estimates for the infimal energy in most homotopy classes, although in some cases they imply it is $0$.  They also do not seem to indicate explicitly how the infimum is determined by the geometry of the spaces in question.  For homotopy classes of mappings of real projective $3$-space, however, the following two-sided estimate shows that the infimum never exceeds the lower bound in (\ref{rpn_inf_ineq_1}) by more than a third: 

\begin{proposition}
\label{rp3_inf_thm}

Let $\Psi$ be a homotopy class of mappings from $(\R P^{3},g_{0})$ to a Riemannian manifold $(M,g)$, $\psi$ the induced homotopy class of mappings of $(\R P^{2},g_{0})$, and $B^{\star}$ the infimal area of mappings $f \in \psi$ as above.  Then $\frac{3\pi}{4} B^{\star} \leq \inf\limits_{F \in \Psi} E_{2}(F) \leq \pi B^{\star}$.  
\end{proposition}

\begin{proof} That $\frac{3\pi}{4} B^{\star} \leq \inf\limits_{F \in \Psi} E_{2}(F)$ is (\ref{rpn_inf_ineq_1}).  To show that $\inf\limits_{F \in \Psi} E_{2}(F) \leq \pi B^{\star}$, note that for a fixed totally geodesic $\mathcal{Q}_{0} \cong \R P^{2}$ in $\R P^{3}$ and any $\varepsilon > 0$, the homotopy extension property for $\R P^{2} \subseteq \R P^{3}$ implies $\Psi$ contains maps $F$ with $E_{2}(F|_{\mathcal{Q}_{0}}) < B^{\star} + \varepsilon$, as in the proof of Theorem \ref{infimum}.  Let $F$ be such a map, $x_{0} \in \R P^{3}$ the unique point at distance $\frac{\pi}{2}$ from $\mathcal{Q}_{0}$, $\widetilde{x}_{0} \in S^{3}$ a point in the preimage of $x_{0}$ via the covering $\tau:S^{3} \rightarrow \R P^{3}$, and $\beta:S^{3} \setminus \lbrace -\widetilde{x}_{0} \rbrace \rightarrow \R^{3}$ the stereographic projection which takes $\widetilde{x}_{0}$ to the origin.  Let $\theta_{t}:S^{3} \rightarrow S^{3}$ be the conformal diffeomorphism of $S^{3}$ given by multiplication by $t$ in the coordinates defined by $\beta$.  An elementary calculation shows $\lim\limits_{t \rightarrow \infty} E_{2}(\theta_{t}) = 0$.  Let $D(\widetilde{x}_{0}) \subseteq S^{3}$ be the open hemisphere centered at $\widetilde{x}_{0}$ and define $\Theta_{t}: \R P^{3} \rightarrow \R P^{3}$ as follows: if $x \in \R P^{3}$ has a preimage $\widetilde{x}$ in $\theta_{t}^{-1}(D(\widetilde{x}_{0}))$, then $\Theta_{t}(x) = \tau \circ \theta_{t}(\widetilde{x})$, otherwise $\Theta_{t}(x)$ is the image of $x$ under the nearest-point projection from $\R P^{3} \setminus x_{0}$ to $\mathcal{Q}_{0}$.  Because $\theta_{t}$ maps $\theta_{t}^{-1}(D(\widetilde{x}_{0}))$ conformally to $D(\widetilde{x}_{0})$ and $\lim\limits_{t \rightarrow \infty} E_{2}(\theta_{t}) = 0$, the energy of $F \circ \Theta_{t}$ on the domain $\tau(\theta_{t}^{-1}(D(\widetilde{x}_{0})))$ goes to $0$ as $t \rightarrow \infty$.  We then have: 
\begin{equation}
\displaystyle \lim\limits_{t \rightarrow \infty} E_{2}(F \circ \Theta_{t}) = \lim\limits_{t \rightarrow \infty} \frac{1}{2} \int\limits_{0}^{\frac{\pi}{2}} \int\limits_{\partial B_{r}(p_{0})} |d(F \circ \Theta_{t})_{x}|^{2} dx dr = \frac{\pi}{2} \times 2 \times E_{2}(F|_{\mathcal{Q}_{0}}) < \pi B^{\star} + \varepsilon. 
\end{equation}

Because $\varepsilon > 0$ was arbitrary, this completes the proof. \end{proof}  

One can give a two-sided estimate in terms of the infimal area $B^{\star}$ in (\ref{rpn_inf_ineq_1}) for the infimum of the energy in homotopy classes of mappings of higher-dimensional real projective spaces.  We establish estimates of this type and characterize maps which realize the associated lower bound for energy in \cite{Hois1}. 

%%%%%%%%%%%%%%%%%%%%%%%%%%%%%%%%%%%%%%%%%%%%%%%%%%%%%%%%%%%%%%%%%%%%%%%%%%%%%%%%%%%%%%%%%%%%%%%%%%%%%%%%%%%%%%%%%%%%%%%%%


\begin{thebibliography}{ABCDE}

\bibitem[Aub13]{Aub1} Thierry Aubin: {\em Some Nonlinear Problems in Riemannian Geometry}, Springer Science and Business Media, 2013.

\bibitem[Be12]{Be1} Arthur L. Besse: {\em Manifolds All of Whose Geodesics are Closed}, Springer Science and Business Media, 2012.

\bibitem[BBdeBR89]{BBdBR1} D. Burns, F. Burstall, P. de Bartolomeis and J. Rawnsley: {\em Stability of Harmonic Maps of K\"ahler Manifolds}, Journal of Differential Geometry, 30(2) (1989), 579-594.

\bibitem[Cr87]{Cr1} Christopher B. Croke: {\em Lower Bounds on the Energy of Maps}, Duke Mathematical Journal 55.4 (1987), 901-908.

\bibitem[CK03]{CK1} Christopher B. Croke and Mikhail Katz: {\em Universal Volume Bounds in Riemannian Manifolds}, Surveys in Differential Geometry 8.1 (2003), 109-137.

\bibitem[ES64]{ES1} James Eells and Joseph H. Sampson: {\em Harmonic Mappings of Riemannian Manifolds}, American Journal of Mathematics 86.1 (1964), 109-160.

\bibitem[Gn62]{Gn} Leon W. Green: {\em Auf Wiedersehensfl\"achen}, Annals of Mathematics 78.2 (1963), 289-299. 

\bibitem[HL87]{HL2} Robert Hardt and Fang‐Hua Lin: {\em Mappings Minimizing the Lp Norm of the Gradient}, Communications on Pure and Applied Mathematics 40.5 (1987), 555-588.

\bibitem[H24]{Hois1} Joseph Hoisington: {\em Energy-Minimizing Mappings of Real Projective Spaces}, ArXiv preprint arXiv:2401.10201 (2024).

\bibitem[KN20]{KN1} Mikhail Katz and Tahl Nowik: {\em A Systolic Inequality With Remainder in the Real Projective Plane}, Open Mathematics 18.1 (2020), 902-906.

\bibitem[KS21]{KS1} Mikhail Katz and St\'ephane Sabourau {\em A Pu-Bonnesen Inequality}, Journal of Geometry 112.2 (2021), 18.

\bibitem[LS73]{LS1} H. Blaine Lawson, Jr. and James Simons: {\em On Stable Currents and Their Application to Global Problems in Real and Complex Geometry}, Annals of Mathematics (1973): 427-450.

\bibitem[Lem78]{Lem1} Luc Lemaire: {\em Applications Harmoniques de Surfaces Riemanniennes}, Journal of Differential Geometry 13.1 (1978), 51-78.

\bibitem[Li70]{Li1} Andr\'e Lichnerowicz: {\em Applications Harmoniques et Vari\'et\'es K\"ahleriennes}, Symp. Math. III, Bologna (1970), 341-402. 

\bibitem[Oh87]{Oh1} Yoshihiro Ohnita: {\em On Pluriharmonicity of Stable Harmonic Maps}, Journal of the London Mathematical Society 2.3 (1987), 563-568.

\bibitem[Oh86]{Oh3} Yoshihiro Ohnita: {\em Stability of Harmonic Maps and Standard Minimal Immersions}, Tohoku Math. J., Second Series 38.2 (1986), 259-267.

\bibitem[Os79]{Os1} Robert Osserman: {\em Bonnesen-style Isoperimetric Inequalities}, The American Mathematical Monthly 86.1 (1979), 1-29.

\bibitem[Pu52]{Pu} Pao Ming Pu: {\em Some Inequalities in Certain Nonorientable Riemannian Manifolds}, Pacific J. Math 2.1 (1952), 55-71.

\bibitem[SU81]{SU} Jonathan Sacks and Karen Uhlenbeck: {\em The Existence of Minimal Immersions of 2-spheres}, Annals of Mathematics 113.1 (1981), 1-24.

\bibitem[We98]{We1} Shishu W. Wei: {\em Representing Homotopy Groups and Spaces of Maps by p-Harmonic Maps},  Indiana University Mathematics Journal 47.2 (1998), 625-670.

\bibitem[Wh86]{Wh1} Brian White: {\em Infima of Energy Functionals in Homotopy Classes of Mappings}, Journal of Differential Geometry 23.2 (1986), 127-142.

\bibitem[Wh88]{Wh3} Brian White: {\em Homotopy Classes in Sobolev Spaces and the Existence of Energy Minimizing Maps}, Acta Mathematica 160.1 (1988), 1-17.

\bibitem[Xin80]{Xi1} Y.L. Xin: {\em Some Results on Stable Harmonic Maps}, Duke Math. J. 47.3 (1980), 609-613.	

\end{thebibliography}
\end{document}